\documentclass[11pt, a4paper]{article}

\usepackage[utf8]{inputenc}
\usepackage[english]{babel}
\usepackage{amsmath,amssymb,amsthm}
\usepackage{xcolor}
\usepackage{tikz}
\usetikzlibrary{arrows.meta, positioning, calc}
\usepackage{caption}
\usepackage{subcaption}
\usepackage{graphicx}
\usepackage[margin=1.2in]{geometry}
\usepackage{hyperref}

\newtheorem{theorem}{Theorem}[section]

\newtheorem{corollary}[theorem]{Corollary}

\newtheorem{conjecture}[theorem]{Conjecture}
\theoremstyle{definition}

\theoremstyle{remark}
\newtheorem{remark}[theorem]{Remark}

\newcommand{\Z}{\mathbb{Z}}

\title{Extinction and Survival for a Generalized Contact Process with Deterministic Cures}

\author{
  Nancy L.~Garcia\thanks{Universidade Estadual de Campinas (UNICAMP), Brazil. Email: nancyg@unicamp.br, denisalu@unicamp.br}
  \and
  Denis A.~Luiz\footnotemark[1]
  \and
  Daniel M.~Machado\thanks{Universidade Federal do ABC (UFABC), Brazil. Email: daniel.miranda@ufabc.edu.br}
}

\date{September 24, 2026} 

\begin{document}

\maketitle

\begin{abstract}
We study a generalized contact process on \(\Z^d\) parameterized by an infection rate \(\lambda\) and a resetting probability \(p \in [0,1]\), modeling deterministic cure times. Once a vertex is infected, its recovery is scheduled exactly one time unit later. Incoming attempts to an already-infected vertex successfully reset its recovery clock with probability \(p\), and are ignored otherwise. This unifies spatial versions of classical Type I (\(p=0\), non-paralyzable) and Type II (\(p=1\), paralyzable) counters. Except in the fully resetting case, deterministic recovery deadlines destroy coordinatewise attractiveness and create a causal shielding effect.

Using a first-moment bound on potential causal chains, we prove that the process dies out from finite configurations whenever \(\lambda<1/(2d)\), uniformly in \(p\in[0,1]\). The same causal-chain bound yields local convergence to the empty configuration from arbitrary initial states. In dimensions \(d \ge 2\), an oriented percolation exploration based on the first infective window establishes global survival for all \(p\in[0,1]\) when \(\lambda>-\log(1-p_c^{\mathrm{or}})\). Finally, in the purely non-resetting case \(p=0\), we derive a delayed identity for the one-site density and prove that this density remains strictly between zero and one at every finite time.

\vspace{0.5cm}
\end{abstract}

\section{Introduction}

We consider a generalized process \((\eta_t)_{t\geq0}\) on \(\Z^d\), with states \(0\) and \(1\) denoting healthy and infected vertices, respectively. Infection attempts are carried by independent rate \(\lambda\) Poisson processes \(\mathcal N^{x,y}\) on oriented nearest-neighbor edges \((x,y)\). An arrow \(x\to y\) at time \(t\) can act only when its source is infected, that is, when \(\eta_{t-}(x)=1\).

The defining feature of this model is a deterministic recovery time governed by a resetting probability parameter \(p \in [0,1]\). When a healthy vertex \(x\) receives an infection attempt, it becomes infected and is scheduled to recover exactly one time unit later. If \(x\) is already infected, a subsequent incoming attempt triggers an independent Bernoulli trial: with probability \(p\), the attempt is accepted, resetting the recovery clock to exactly one time unit after the arrival; with probability \(1-p\), the attempt is ignored.

Unless explicitly stated otherwise, every site infected at time \(0\) is scheduled to recover at time \(1\).

This framework provides a spatial analogue of classical models from the theory of counters and renewal processes \cite{feller1948, takacs1958}:
\begin{itemize}
    \item \(\mathbf{p=0}\) \textbf{(Type I Counter):} The Contact Process with Deterministic Cure (CPDC) or pure non-resetting case. An infected vertex ignores all subsequent attempts and recovers exactly one time unit after its initial infection. This corresponds to the non-paralyzable counter model.  In dimension \(d=1\), after a deterministic rescaling of time, this coincides with the fixed-lifetime non-Markovian contact process considered in \cite{gerami2002criticality}.
    \item \(\mathbf{p=1}\) \textbf{(Type II Counter):} The Contact Process with Restarting Deterministic Cure (CPRDC) or the full resetting case. Every incoming attempt successfully resets the clock. A vertex recovers only after a full interval of length \(1\) without receiving any effective infection attempt from an infected neighbor. This corresponds to the paralyzable counter model.
\end{itemize}

A fundamental difficulty for the pure CPDC (\(p=0\)) and the intermediate models (\(0<p<1\)) is the failure of attractiveness. An earlier infection may impose an earlier recovery time and thereby prevent subsequent reset attempts from taking effect, as illustrated in Figure~\ref{fig:graphical_coupling}. Consequently, a larger initial infected set need not produce a larger infected set at later times. The CPDC is not attractive, and the natural graphical coupling is not
pathwise monotone in either \(\lambda\) or \(p\).

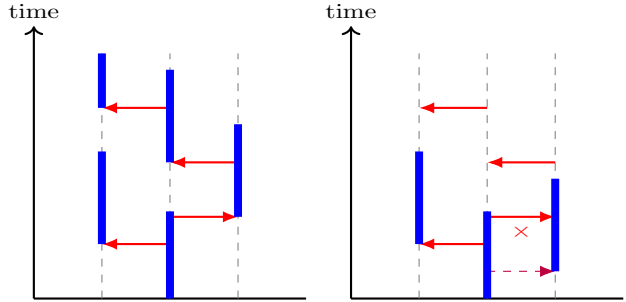
\begin{figure}[!ht]
    \centering
    \begin{tikzpicture}[scale=0.6, transform shape, xscale=1.5, yscale=1.2]
        \draw[->, thick] (0,0) -- (0,5) node[above] {time};
        \draw[-, thick] (0,0) -- (4,0);
        \foreach \x in {1,2,3} {
                \draw[dashed, gray] (\x,0) -- (\x,4.5);
            }
        \draw[-{Latex[length=2mm]}, red, thick] (2,1) -- (1,1);
        \draw[-{Latex[length=2mm]}, red, thick] (3,2.5) -- (2,2.5);
        \draw[-{Latex[length=2mm]}, red, thick] (2,1.5) -- (3,1.5);
        \draw[-{Latex[length=2mm]}, red, thick] (2,3.5) -- (1,3.5);
        \draw[line width=3pt, blue] (1,1) -- (1,2.7);
        \draw[line width=3pt, blue] (2,2.5) -- (2,4.2);
        \draw[line width=3pt, blue] (2,0) -- (2,1.6);
        \draw[line width=3pt, blue] (3,1.5) -- (3,3.2);
        \draw[line width=3pt, blue] (1,3.5) -- (1,4.5);
    \end{tikzpicture}
    \begin{tikzpicture}[scale=0.6, transform shape, xscale=1.5, yscale=1.2]
        \draw[->, thick] (0,0) -- (0,5) node[above] {time};
        \draw[-, thick] (0,0) -- (4,0);
        \foreach \x in {1,2,3} {
                \draw[dashed, gray] (\x,0) -- (\x,4.5);
            }
        \draw[-{Latex[length=2mm]}, red, thick] (2,1) -- (1,1);
        \draw[-{Latex[length=2mm]}, red, thick] (3,2.5) -- (2,2.5);
        \draw[-{Latex[length=2mm]}, red, thick] (2,1.5) -- (3,1.5) node[midway,below] {$\times$};
        \draw[-{Latex[length=2mm]}, red, thick] (2,3.5) -- (1,3.5);
        \draw[-{Latex[length=2mm]}, purple, dashed] (2,0.5) -- (3,0.5);
        \draw[line width=3pt, blue] (1,1) -- (1,2.7);
        \draw[line width=3pt, blue] (2,0) -- (2,1.6);
        \draw[line width=3pt, blue] (3,0.5) -- (3,2.2);
    \end{tikzpicture}
    \caption{Graphical construction of the non-resetting CPDC ($p=0$). Thick blue lines indicate deterministic infection periods of length 1, while red arrows represent rate $\lambda$ infection attempts. Note the loss of attractiveness: adding the dashed early infection arrow causes the receiving site to be actively infected when the subsequent arrow arrives. The site therefore ignores this later attempt (marked with $\times$), ultimately reducing the total epidemic duration in this realization. For $p>0$, this shielding effect occurs whenever the independent Bernoulli reset trial fails.}
    \label{fig:graphical_coupling}
\end{figure}

The fully resetting process (\(p=1\)) dominates every process with parameter \(p\) under the natural graphical coupling, for every $\lambda>0$. Indeed, every reset accepted for \(p\) is
also accepted for \(p=1\), while additional resets can only extend the
recovery deadline in the latter. This ordering is preserved at each
graphical event, and hence \(\eta_t^p\subseteq\eta_t^1\) for all \(t\geq0\).

For fixed rate $\lambda>0$ and probability parameter $p\in[0,1]$, we denote by CPDC\((\lambda,p)\) the generalized contact process with deterministic cures, reserving CPDC and CPRDC for the endpoint cases \(p=0\) and \(p=1\), respectively.
 For a finite initial infected set, let the extinction time be
 \(
 \tau:=\inf\{t\geq0:\eta_{t}^p=\varnothing\}.
 \)
 Since the empty configuration is absorbing, we say that the process survives globally when \(\mathbb P(\tau=\infty)>0\).

We now collect the main statements of the paper.

\begin{theorem}[Main Results]
\label{thm:main_results}
Consider the CPDC\((\lambda,p)\) on \(\mathbb Z^d\).
\begin{enumerate}
    \item \textbf{Extinction from finite initial configurations.} If
    \[
        \lambda<\frac{1}{2d},
    \]
    then the process dies out almost surely from every finite initial configuration, for every \(p\in[0,1]\).

    \item \textbf{Local extinction from arbitrary initial configurations.}
    Under the same condition, for every initial binary configuration \(\xi\in\{0,1\}^{\mathbb Z^d}\),
    \[
        \eta_t^{\xi}\Longrightarrow\delta_\emptyset
        \qquad\text{as }t\to\infty
    \]
    in the local topology.

    \item \textbf{Global survival.} Assume \(d\geq2\). If
    \[
        \lambda>-\log(1-p_c^{\mathrm{or}}),
    \]
    where \(p_c^{\mathrm{or}}\) is the critical parameter for oriented bond percolation on \(\Z_+^2\) \cite{durrett1984oriented}, the process survives globally with positive probability from every finite nonempty initial configuration, for all \(p \in [0,1]\).

    \item \textbf{Delayed density identity and strict confinement (\(p=0\)).} For the purely non-resetting case (\(p=0\)), let the process start from a Bernoulli product measure with density \(\rho_{0}\in(0,1)\). For any fixed nearest-neighbor pair \(y\sim x\), let \(\rho(t)=\mathbb P[\eta_t(x)=1]\) and \(\rho_{10}(t)=\mathbb P[\eta_t(y)=1,\eta_t(x)=0]\). Then \(\rho\) is locally absolutely continuous on \((1,\infty)\) and, for every \(\lambda>0\) and Lebesgue-a.e. \(t>1\),
    \[
        \rho'(t)=2d\lambda \rho_{10}(t)-2d\lambda \rho_{10}(t-1).
    \]
    Moreover, the density is strictly confined: \(0<\rho(t)<1\) for every \(t>0\).
\end{enumerate}
\end{theorem}

\begin{remark}[Space-time coupling for \(p>0\) and \(d=1\)]\label{rmk:fractional_window}
A simpler space-time coupling establishes survival for \(p>0\), as illustrated in Figure~\ref{fig:global_survival_coupling}. Splitting time into half-integer intervals, retain only arrows whose
Bernoulli marks are successful. If the source remains active throughout
\((a,a+1/2]\), any retained arrow from it during this interval either
infects a healthy target or resets an infected target, and therefore
guarantees that the target remains active throughout the subsequent
half-window \((a+1/2,a+1]\). This directly embeds discrete-time oriented percolation (OP) on \(\Z^d \times \Z_+\) with parameter \(1-e^{-p\lambda/2}\). Because time acts as a coordinate in the oriented percolation lattice, this yields a survival regime even in one spatial dimension (\(d=1\)).

However, this fractional coupling fails completely for the pure CPDC (\(p=0\)), as earlier infections can shield the vertex and expire before the required transmission window. It also yields only a weaker sufficient condition for survival. The dynamical exploration in Section~\ref{sec:survival} overcomes both limitations solving the \(p=0\) case and recovering a better parameter \(1-e^{-\lambda}\), at the cost of requiring \(d \geq 2\) spatial dimensions.
\end{remark}

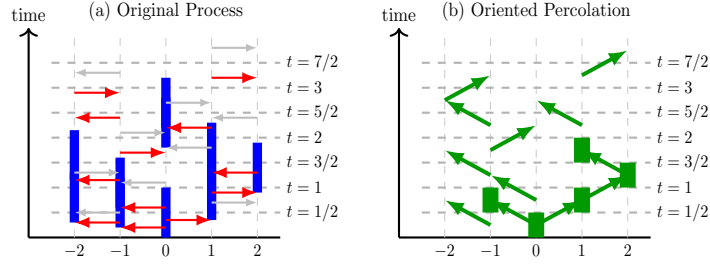
\begin{figure}[h!]
    \colorlet{myblue}{blue}
    \colorlet{myred}{red}
    \colorlet{boldgreen}{green!60!black}
    \colorlet{poissongray}{gray!50} 
    \colorlet{softgrid}{gray!40}
    \colorlet{tline}{gray!60}
    \centering
    \begin{subfigure}[b]{0.32\textwidth}
        \centering
        \begin{tikzpicture}[scale=0.55, transform shape, xscale=1.1, yscale=1.2]

            \draw[->, thick] (0,0) -- (0,4.2) node[above] {time};
            \draw[-, thick] (0,0) -- (5.5,0);
            
            \node[below] at (1,0) {\small $-2$};
            \node[below] at (2,0) {\small $-1$};
            \node[below] at (3,0) {\small $0$}; 
            \node[below] at (4,0) {\small $1$};
            \node[below] at (5,0) {\small $2$};
            \foreach \x in {1,...,5} {
                \draw[dashed, softgrid] (\x,0) -- (\x,4);
            }
            
            \foreach \x in {1/2,1,3/2,2,5/2,3,7/2} {
                \draw[dashed, tline, thick] (0.5,\x) -- (5.5,\x) node[right, text=black] {\small $t=\x$};
            }

            \draw[line width=3.5pt, myblue] (3,0) -- (3,1);
            \draw[line width=3.5pt, myblue] (2,0.2) -- (2,1.2);
            \draw[line width=3.5pt, myblue] (4,0.35) -- (4,1.35);
            \draw[line width=3.5pt, myblue] (2,0.6) -- (2,1.6);
            \draw[line width=3.5pt, myblue] (1,0.3) -- (1,1.3);
            \draw[line width=3.5pt, myblue] (1,1.15) -- (1,2.15);
            \draw[line width=3.5pt, myblue] (5,0.9) -- (5,1.9);
            \draw[line width=3.5pt, myblue] (4,1.3) -- (4,2.3);
            \draw[line width=3.5pt, myblue] (3,1.8) -- (3,2.8);
            \draw[line width=3.5pt, myblue] (3,2.2) -- (3,3.2);

            \draw[-{Latex[length=1.5mm]}, poissongray, thick] (1,1.3) -- (2,1.3);
            \draw[-{Latex[length=1.5mm]}, poissongray, thick] (2,0.5) -- (1,0.5);
            \draw[-{Latex[length=1.5mm]}, poissongray, thick] (2,2.1) -- (3,2.1);
            \draw[-{Latex[length=1.5mm]}, poissongray, thick] (2,3.3) -- (1,3.3);
            \draw[-{Latex[length=1.5mm]}, poissongray, thick] (3,1.1) -- (2,1.1);
            \draw[-{Latex[length=1.5mm]}, poissongray, thick] (3,2.7) -- (4,2.7);
            \draw[-{Latex[length=1.5mm]}, poissongray, thick] (4,0.7) -- (5,0.7);
            \draw[-{Latex[length=1.5mm]}, poissongray, thick] (4,1.8) -- (3,1.8);
            \draw[-{Latex[length=1.5mm]}, poissongray, thick] (4,3.8) -- (5,3.8);
            \draw[-{Latex[length=1.5mm]}, poissongray, thick] (5,2.4) -- (4,2.4);

            \draw[-{Latex[length=2mm]}, myred, thick] (3,0.35) -- (4,0.35);
            \draw[-{Latex[length=2mm]}, myred, thick] (3,0.2) -- (2,0.2); 
            \draw[-{Latex[length=2mm]}, myred, thick] (5,1.3) -- (4,1.3);
            \draw[-{Latex[length=2mm]}, myred, thick] (2,0.3) -- (1,0.3); 
            \draw[-{Latex[length=2mm]}, myred, thick] (4,2.2) -- (3,2.2); 
            \draw[-{Latex[length=2mm]}, myred, thick] (2,1.15) -- (1,1.15); 
            \draw[-{Latex[length=2mm]}, myred, thick] (2,2.4) -- (1,2.4); 
            \draw[-{Latex[length=2mm]}, myred, thick] (4,3.2) -- (5,3.2); 
            \draw[-{Latex[length=2mm]}, myred, thick] (1,2.9) -- (2,2.9); 
            \draw[-{Latex[length=2mm]}, myred, thick] (2,1.7) -- (3,1.7); 
            \draw[-{Latex[length=2mm]}, myred, thick] (3,0.6) -- (2,0.6); 
            \draw[-{Latex[length=2mm]}, myred, thick] (4,0.9) -- (5,0.9); 

            \node[above] at (3,4.2) {(a) Original Process};
        \end{tikzpicture}
    \end{subfigure}
    \begin{subfigure}[b]{0.32\textwidth}
        \centering
        \begin{tikzpicture}[scale=0.55, transform shape, xscale=1.1, yscale=1.2]
            % Axis
            \draw[->, thick] (0,0) -- (0,4.2) node[above] {time};
            \draw[-, thick] (0,0) -- (5.5,0);
            
            \node[below] at (1,0) {\small $-2$};
            \node[below] at (2,0) {\small $-1$};
            \node[below] at (3,0) {\small $0$}; 
            \node[below] at (4,0) {\small $1$};
            \node[below] at (5,0) {\small $2$};
            \foreach \x in {1,...,5} {
                \draw[dashed, softgrid] (\x,0) -- (\x,4);
            }
            
            \foreach \x in {1/2,1,3/2,2,5/2,3,7/2} {
                \draw[dashed, tline, thick] (0.5,\x) -- (5.5,\x) node[right, text=black] {\small $t=\x$};
            }

            \draw[line width=6pt, boldgreen] (3,0) -- (3,0.5);
            \draw[line width=6pt, boldgreen] (4,0.5) -- (4,1.0);
            \draw[line width=6pt, boldgreen] (2,0.5) -- (2,1.0);
            \draw[line width=6pt, boldgreen] (5,1.0) -- (5,1.5);
            \draw[line width=6pt, boldgreen] (4,1.5) -- (4,2.0);

            \draw[-{Latex[length=2.5mm]}, boldgreen, line width=1.5pt] (3,0.25) -- (4,0.75);
            \draw[-{Latex[length=2.5mm]}, boldgreen, line width=1.5pt] (3,0.25) -- (2,0.75);
            \draw[-{Latex[length=2.5mm]}, boldgreen, line width=1.5pt] (5,1.25) -- (4,1.75);
            \draw[-{Latex[length=2.5mm]}, boldgreen, line width=1.5pt] (2,0.25) -- (1,0.75);
            \draw[-{Latex[length=2.5mm]}, boldgreen, line width=1.5pt] (4,2.25) -- (3,2.75);
            \draw[-{Latex[length=2.5mm]}, boldgreen, line width=1.5pt] (2,1.25) -- (1,1.75);
            \draw[-{Latex[length=2.5mm]}, boldgreen, line width=1.5pt] (2,2.25) -- (1,2.75);
            \draw[-{Latex[length=2.5mm]}, boldgreen, line width=1.5pt] (4,3.25) -- (5,3.75);
            \draw[-{Latex[length=2.5mm]}, boldgreen, line width=1.5pt] (1,2.75) -- (2,3.25);
            \draw[-{Latex[length=2.5mm]}, boldgreen, line width=1.5pt] (2,1.75) -- (3,2.25);
            \draw[-{Latex[length=2.5mm]}, boldgreen, line width=1.5pt] (3,0.75) -- (2,1.25);
            \draw[-{Latex[length=2.5mm]}, boldgreen, line width=1.5pt] (4,0.75) -- (5,1.25);

            \node[above] at (3,4.2) {(b) Oriented Percolation};
        \end{tikzpicture}
    \end{subfigure}
    \caption{Graphical construction of the survival coupling in Remark~\ref{rmk:fractional_window}. \textbf{(a)} The CPDC started from a single infection at \(x=0\). Gray arrows carry unsuccessful Bernoulli marks, whereas red arrows carry successful marks. \textbf{(b)} The associated oriented percolation configuration. Every green oriented path starting from the initial seed corresponds to an infection path in the original process.}
    \label{fig:global_survival_coupling}
\end{figure}

\begin{remark}[Uniformity in \(p\)]\label{rmk:uniformity}
    The proof counts every unmarked graphical arrow as potentially causal and therefore disregards both the target state and the Bernoulli mark. The bound \(\lambda<1/(2d)\) is consequently uniform in \(p\), including the fully resetting case \(p=1\).
\end{remark}

The preceding extinction and survival arguments also remain valid when the resetting probability varies with time. That is, let \(\mathbf p:\mathbb R_+\to[0,1]\) be a Borel measurable function and consider a modification of the contact process with deterministic cure such that a reinfection attempt arriving at time \(t\) is accepted with probability \(\mathbf p(t)\). We refer to this process as CPDC$(\lambda,\mathbf{p})$.

\begin{corollary}
Let \(\mathbf p:\mathbb R_+\to[0,1]\) be Borel measurable, and consider the CPDC\((\lambda,\mathbf p)\). Items~\(1\)--\(3\) of Theorem~\ref{thm:main_results} remain valid, with the same bounds, for this time-dependent resetting rule.
\end{corollary}

\begin{proof}
The causal-chain argument disregards the Bernoulli marks, while the survival construction uses only the first infection interval of each vertex. Neither argument requires the resetting probability to be constant.
\end{proof}

The attractiveness of CPDC\((\lambda,1)\), together with
Theorem~\ref{thm:main_results}, yields the existence of a phase transition
in \(\lambda\) for this value of \(p\). Motivated by the numerical results in Figure~\ref{fig:interpolaharm}, we propose the following conjecture for the critical curve.

\begin{conjecture}
For any fixed \(p\in[0,1]\), there is a phase transition in \(\lambda\); that is, there exists a parameter \(\lambda_c(p)\) such that the CPDC\((\lambda,p)\) becomes extinct for $\lambda<\lambda_{c}(p)$ and survives with positive probability for $\lambda>\lambda_{c}(p)$. Moreover, the critical curve is given by the weighted harmonic interpolation
\[
	\lambda_c(p)
	=
	\frac{\lambda_c(0)\lambda_c(1)}
	{(1-p)\lambda_c(1)+p\lambda_c(0)}.
\]
Numerically \(\lambda_c(0)\approx0.40755\) and \(\lambda_c(1)\approx0.33247\) for \(\mathbb{Z}^{2}\).
\end{conjecture}

\begin{figure}[!h]
    \centering
    \includegraphics[width=0.65\linewidth]{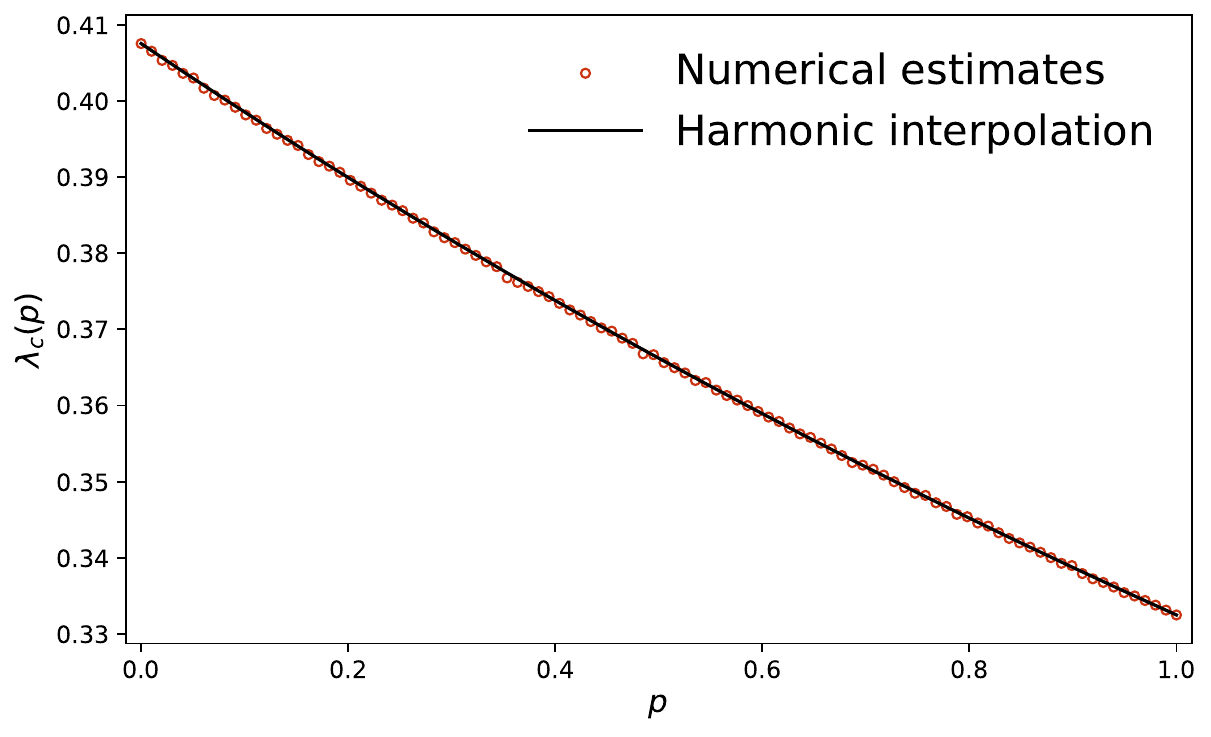}
\caption{Simulation-based estimates of the critical curve
\(\lambda_c(p)\) on \(\mathbb Z^2\), together with the
conjectured harmonic interpolation. A bootstrap goodness-of-fit test did
not reject the weighted harmonic interpolation at the \(5\%\) level.}
\label{fig:interpolaharm}
\end{figure}

Recent developments in interacting particle systems have explored non-Markovian dynamics, particularly processes under random environments, multi-stage infections, or dynamic edges \cite{fontes2019renewal, hilario2022results, lanchier2026contact}. While models with heavy-tailed renewals often preserve monotonicity, our model isolates the distinctive effects of deterministic recovery times.

The paper is organized as follows. Section \ref{sec:extinction} gives the graphical construction and the subcritical arguments via causal-chain bounds and perfect simulation. Section \ref{sec:survival} proves global survival via an oriented percolation argument using only the first infective window. Section \ref{sec:density} derives the delayed density identity for the case \(p=0\) and explains why the same one-window representation does not extend directly to \(p>0\).

\section{Extinction: Graphical Construction and Causal Chains}\label{sec:extinction}

To prove extinction, we use the usual Harris graphical construction. Let \(\{\mathcal N^{x,y}:x\sim y\}\) be independent Poisson processes of rate \(\lambda\). To each arrival point at time \(t\), we attach an independent uniform random variable \(U_{t}^{x,y} \sim \mathcal{U}(0,1)\). An arrow \(x\to y\) is effective only if \(\eta_{t-}(x)=1\). If an effective arrow arrives at a healthy vertex, it becomes infected and its recovery is scheduled at \(t+1\). If the vertex is already infected, the attempt successfully resets the clock to \(t+1\) if, and only if, \(U_{t}^{x,y} \le p\).

\subsection{Causal-chain bound}

For a finite set \(B\subset\mathbb Z^d\), let \(C_n(B)\) be the number of sequences
\[
	x_0\in B,\qquad x_i\sim x_{i-1},\qquad
	0=s_0<s_1<\cdots<s_n,
\]
such that
\[
	s_i\in\mathcal N^{x_{i-1},x_i},
	\qquad 0<s_i-s_{i-1}\leq1.
\]
We set \(C_0(B)=|B|\). For each nearest-neighbor walk, the change of variables \(u_i=s_i-s_{i-1}\) identifies the admissible time region with \((0,1]^n\). Hence the multivariate Campbell formula gives
\[
\begin{aligned}
	\mathbb E[C_n(B)]
	&=
	\sum_{x_0\in B}
	\sum_{x_1\sim x_0}\cdots
	\sum_{x_n\sim x_{n-1}}
	\int_{(0,1]^n}\lambda^n\,du_1\cdots du_n \\
	&=|B|(2d\lambda)^n.
\end{aligned}
\]
The formula remains valid when an oriented edge is traversed more than once, since the arrow times are strictly ordered.

\subsection{Finite initial configurations}

\begin{proof}[Proof of global extinction from finite initial configurations]
Fix a finite initial infected set \(A\subset\mathbb Z^d\), with residual recovery times in \((0,1]\). Call an effective arrow \emph{deadline-setting} if it infects a healthy target or produces an accepted reset. In either case, an arrow at time \(s\) sets the target's recovery deadline equal to \(s+1\).

Tracing the ancestry of any infected space--time point backward through deadline-setting arrows yields one of the causal chains counted above. This backward tracing is non-explosive (see \cite{bhattacharya1990stochastic}): if \(R_n(T)\) denotes the number of graphical paths of length \(n\) starting from \(A\) and contained in \([0,T]\), then
\[
	\mathbb E[R_n(T)]
	=
	|A|\frac{(2d\lambda T)^n}{n!},
\]
and therefore \(\mathbb E[\sum_{n\geq0}R_n(T)]=|A|e^{2d\lambda T}<\infty\). 

If \(2d\lambda<1\), Tonelli's theorem gives
\[
	\mathbb E\left[\sum_{n\geq0}C_n(A)\right]
	=
	\frac{|A|}{1-2d\lambda}<\infty.
\]
Thus only finitely many deadline-setting arrows are causally reachable from \(A\). If \(S\) is the time of the last one, with \(S=0\) when none exists, every recovery deadline is at most \(S+1\). Consequently,
\[
	\eta_t^A=\varnothing
	\qquad\text{for all }t>S+1.
\]
Since the argument counts all unmarked arrows, it applies uniformly to every \(p\in[0,1]\).
\end{proof}

\subsection{Backward Construction and Perfect Simulation}

Fix a finite set \(K\subset\mathbb Z^d\). Starting from \(K\times\{0\}\), explore backward all sequences
\[
	x_0\in K,\qquad x_i\sim x_{i-1},\qquad
	0=s_0<s_1<\cdots<s_n,
\]
such that
\[
	-s_i\in\mathcal N^{x_i,x_{i-1}},
	\qquad 0<s_i-s_{i-1}\leq1.
\]

This construction uses only the unmarked Poisson environment. It plays the role of the \textit{free process} in the perfect simulation scheme of Ferrari, Fernández, and Garcia \cite{ferrari2002perfect}. Because it disregards the target states and the Bernoulli reset marks, this free environment strictly dominates our process and contains every possible causal ancestry. The true process can be recovered by a forward thinning of this free process. Let $\mathcal A^K$ denote the resulting free clan of ancestors. Figure~\ref{fig:clan_ancestors} illustrates this backward construction.

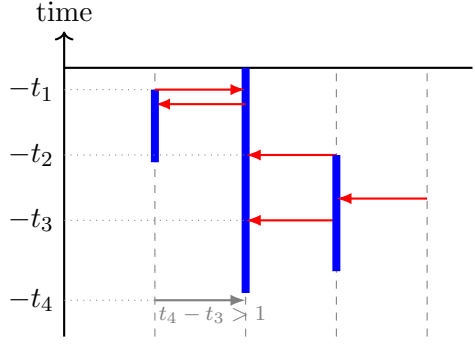
\begin{figure}[!ht]
    \centering
    \begin{tikzpicture}[scale=0.8,xscale=1.5, yscale=1.2]
        \draw[->, thick] (0,0) -- (0,4.2) node[above] {time};
        \draw[-, thick] (0,3.7) -- (4.5,3.7);

        \foreach \x/\name in {1/x_{-1}, 2/x, 3/x_{1}, 4/x_2} {
            \draw[dashed, gray] (\x,0) -- (\x,3.7);
        }

        \draw[line width=3pt, blue] (2,3.7) -- (2,0.6);

        \draw[-{Latex[length=2mm]}, red, thick] (2,3.2) -- (1,3.2);
        \draw[-{Latex[length=2mm]}, red, thick] (3,2.5) -- (2,2.5);
        \draw[-{Latex[length=2mm]}, red, thick] (3,1.6) -- (2,1.6);

        \draw[-{Latex[length=2mm]}, gray, thick] (1,0.5) -- (2,0.5);
        \node[gray, right] at (0.93, 0.3) {\scriptsize $t_4 - t_3 > 1$};

        \draw[line width=3pt, blue] (1,3.4) -- (1,2.4);
        \draw[-{Latex[length=2mm]}, red, thick] (1,3.4) -- (2,3.4);
        
        \draw[line width=3pt, blue] (3,2.5) -- (3,0.9);
        \draw[-{Latex[length=2mm]}, red, thick] (4,1.9) -- (3,1.9);

        \draw[dotted, gray] (0,3.4) -- (1,3.4);
        \node[left] at (0,3.4) {$-t_1$};

        \draw[dotted, gray] (0,2.5) -- (2,2.5);
        \node[left] at (0,2.5) {$-t_2$};

        \draw[dotted, gray] (0,1.6) -- (2,1.6);
        \node[left] at (0,1.6) {$-t_3$};

        \draw[dotted, gray] (0,0.5) -- (1,0.5);
        \node[left] at (0,0.5) {$-t_4$};

    \end{tikzpicture}
    \caption{Backward graphical construction of the free clan of ancestors, corresponding to the potential causal chains. Thick blue lines represent the active backward search intervals of length 1, while red arrows indicate valid deadline-setting arrivals. The gap between arrivals $-t_3$ and $-t_4$ is $1.1 > 1$. Therefore, the search interval on $x$ expires at $-t_3 - 1$, and the arrival at $-t_4$ is ignored. The resulting break in the causal chain illustrates the finite temporal depth of the clan.}
    \label{fig:clan_ancestors}
\end{figure}

Let \(C_n^-(K)\) denote the number of backward sequences of length \(n\).
By time reversal and the preceding computation,
\[
	\mathbb E[C_n^-(K)]=|K|(2d\lambda)^n.
\]
Hence, if \(2d\lambda<1\), then
\[
	\mathbb E\left[\sum_{n\geq0}C_n^-(K)\right]
	=
	\frac{|K|}{1-2d\lambda}<\infty.
\]
Thus \(\sum_{n\geq0}C_n^-(K)<\infty\) almost surely, and hence
\(\mathcal A^K\) is almost surely finite.
Define its maximal backward temporal depth by
\[
	D_K:=1+\sup\{s_n:\text{there exists a backward causal chain ending at }-s_n\},
\]
with the supremum equal to \(0\) when there is no arrow.
The additional \(1\) accounts for the maximal residual lifetime.
Then \(D_K<\infty\) almost surely and
\[
	\mathbb P(D_K>t)\longrightarrow0.
\]

Consider now the process started at time \(0\) from an arbitrary configuration \(\xi\in\{0,1\}^{\mathbb{Z}^{d}}\). If some site of \(K\) is infected at time \(t\), tracing its successive recovery deadlines backward produces a potential causal chain reaching time \(0\). By time translation,
\[
	\sup_{\xi}
	\mathbb P\bigl(\eta_t^{\xi}\cap K\neq\varnothing\bigr)
	\leq
	\mathbb P(D_K\geq t)
	\longrightarrow0.
\]
This proves local extinction in the sense of convergence to the empty configuration in the local topology.

The same construction gives perfect simulation in every finite window in the sense of \cite{ferrari2002perfect}. Because the free clan of ancestors is almost surely finite, the backward search terminates at the finite temporal depth $D_K$. Evaluating the process forward from any time prior to $-D_K$, every admissible initial state yields the empty state on $K$ at time $0$. Since there are no spontaneous infections, the exact stationary state on $K$ is empty.

If \(\mu\) is an invariant probability measure for the augmented process, invariance and the preceding uniform bound give, for every finite \(K\),
\[
	\mu(\eta\cap K\neq\varnothing)
	\leq \mathbb P(D_K\geq t)\longrightarrow0.
\]
Thus \(\mu=\delta_\emptyset\), which is therefore the unique invariant probability measure.

\section{Global Survival: The Oriented Percolation Argument}
\label{sec:survival}

\begin{proof}[Proof of global survival from finite initial configurations]
Assume \(d\geq2\), fix \(\lambda>0\) and \(p\in[0,1]\), and let
\(B\subset\Z^d\) be finite and nonempty. Write \(\mathbb P_B\) for the law
of the process started from \(B\), and let \(\tau\) be its extinction time,
with \(\tau=\infty\) if the infected set never becomes empty. Choose
\(z\in B\) and two distinct standard basis vectors \(e_1,e_2\); this is the
only place where \(d\geq2\) is used.

For \(x\in\Z^d\), let \(S_1^x\) be its first infection time, with
\(S_1^x=\infty\) if \(x\) is never infected. Since an accepted reset only
postpones recovery, the first infection episode of \(x\), when it exists,
contains
\[
    [S_1^x,S_1^x+1).
\]

We use an episode-age graphical construction. For every \(x\in\Z^d\),
\(k\geq1\), and neighbor \(y\) of \(x\), let
\(\mathcal N^{x,k,y}\) be an independent marked Poisson process on
\((0,\infty)\times\{0,1\}\) with intensity
\[
    \lambda\,ds\otimes\bigl((1-p)\delta_0+p\delta_1\bigr).
\]
The index \(k\) labels the infection episodes of \(x\). If the \(k\)-th
episode begins at time \(T_k^x\), then \(T_1^x=S_1^x\) whenever
\(S_1^x<\infty\). A point \((s,u)\) of \(\mathcal N^{x,k,y}\) represents a
transmission attempt from \(x\) to \(y\) at time \(T_k^x+s\), provided that
the episode is still active. If \(y\) is already infected, \(u\) determines
whether the corresponding reset is accepted.

All clock families are assigned in advance. When the \(k\)-th episode of
\(x\) begins, the revealed history determines the index \((x,k)\), while
the families \((\mathcal N^{x,k,y})_{y\sim x}\) have not been inspected.
By product independence, conditionally on this history, these families
have their original joint law and are independent of the previously
inspected coordinates. Activating them and proceeding chronologically
therefore gives the CPDC\((\lambda,p)\) up to its possible explosion time
\(\zeta\).

We briefly exclude explosion. For \(t<\zeta\), let \(D_t\) be the number of
distinct vertices infected by time \(t\). At most \(D_t\) vertices are
infected at time \(t\), each generating transmission attempts at total rate
\(2d\lambda\). Thus \(D\) is dominated, up to \(\zeta\), by a Yule process
\(Y\) with birth rate \(2d\lambda\) per particle, started from \(|B|\)
particles. Since \(Y\) does not explode, on \(\{\zeta<\infty\}\),
\[
    D_{\zeta-}\leq Y_\zeta<\infty.
\]
Successive episodes of the same vertex begin at least one unit of time
apart. Hence, if \(\zeta<\infty\), only finitely many episode-clock families
are activated before \(\zeta\). Each such family contributes finitely many
points in any bounded age interval, and each episode contributes at most one
recovery time. Thus only finitely many events can occur before \(\zeta\),
contradicting the definition of the explosion time. Therefore
\(\zeta=\infty\) almost surely.

For \(x\in\Z^d\) and \(i\in\{1,2\}\), define
\[
    W_{x,i}
    :=
    \inf\bigl\{
        s>0:
        \mathcal N^{x,1,x+e_i}\bigl((0,s]\times\{0,1\}\bigr)>0
    \bigr\},
    \qquad \inf\varnothing:=\infty.
\]
Since the mark measure is a probability measure, the time projection of
each \(\mathcal N^{x,1,x+e_i}\) is a rate-\(\lambda\) Poisson process.
These are distinct coordinates of the product space, so \(W_{x,i}\) is
defined whether or not \(x\) is ever infected and
\[
    (W_{x,i})_{x\in\Z^d,\,i\in\{1,2\}}
\]
is an i.i.d.\ family of \(\operatorname{Exp}(\lambda)\) random variables.

Consider the oriented sublattice
\[
    \mathbb L_+
    :=
    \{z+ae_1+be_2:a,b\in\Z_+\}.
\]
Declare the oriented edge \(x\to x+e_i\) open when \(W_{x,i}<1\). This
defines oriented bond percolation on
\(\mathbb L_+\cong\Z_+^2\) with parameter
\[
    q_\lambda=1-e^{-\lambda}.
\]
Let \(\mathcal C(z)\) denote the open oriented cluster of \(z\). We claim
that
\[
    \mathcal C(z)
    \subseteq
    \{x\in\Z^d:S_1^x<\infty\}.
\]

Indeed, argue by induction on the level \(a+b\) of
\(x=z+ae_1+be_2\). Since \(z\in B\), we have \(S_1^z=0\). If
\(x+e_i\) is reached from \(x\in\mathcal C(z)\) by an open edge, then
\(S_1^x<\infty\) by induction and \(W_{x,i}<1\). During the first episode
of \(x\), the clock associated with \(x\to x+e_i\) therefore rings at time
\[
    S_1^x+W_{x,i}<S_1^x+1,
\]
while \(x\) is still infected. Consequently,
\[
    S_1^{x+e_i}\leq S_1^x+W_{x,i}<\infty,
\]
which proves the claim.

Let \(p_c^{\mathrm{or}}\) be the critical parameter for oriented bond
percolation on \(\Z_+^2\) \cite{durrett1984oriented}. If
\[
    q_\lambda>p_c^{\mathrm{or}},
\]
then
\[
    \mathbb P_B\bigl(|\mathcal C(z)|=\infty\bigr)>0.
\]
On this event, infinitely many distinct vertices are eventually infected.
Since \(D_t<\infty\) for every finite \(t\), their first infection times are
unbounded. As the empty configuration is absorbing, this implies
\(\tau=\infty\). Therefore,
\[
    \mathbb P_B(\tau=\infty)
    \geq
    \mathbb P_B\bigl(|\mathcal C(z)|=\infty\bigr)
    >0.
\]

Thus the CPDC\((\lambda,p)\) survives globally with positive probability
whenever
\[
    \lambda>-\log\bigl(1-p_c^{\mathrm{or}}\bigr).
\]
The bound is uniform in \(p\in[0,1]\) and in the finite nonempty set \(B\),
since the argument uses only that resets cannot shorten the first infective
window.
\end{proof}

\section{Density Identity and Finite-Time Nondegeneracy}\label{sec:density}

Let \((\eta_t)_{t \ge 0}\) be the pure CPDC (\(p=0\)) initiated with a product Bernoulli measure \(\nu_{\rho_{0}}\), where \(\rho_{0} \in (0,1)\). By translation invariance and the symmetries of \(\mathbb Z^d\), the
macroscopic density
\(\rho(t)=\mathbb P[\eta_t(x)=1]\) is independent of the site \(x\).
Moreover, for any ordered nearest-neighbor pair \(y\sim x\), define
\[
\rho_{10}(t)=\mathbb P[\eta_t(y)=1,\eta_t(x)=0],
\]
which is independent of the chosen ordered pair.

\begin{proof}[Proof of Theorem~\ref{thm:main_results}, delayed density identity]
Let \(I_x(a,b]\) denote the number of successful infection arrows into \(x\) during \((a,b]\). Because \(p=0\), every effective infection produces an infection interval of length exactly \(1\), and no intervals overlap. Thus, for \(t>1\), \(\eta_t(x)=I_x(t-1,t]\) almost surely. 

The predictable intensity of the effective infection process at \(x\) is
\begin{equation*}
    \lambda\mathbf 1_{\{\eta_{s-}(x)=0\}} \sum_{y\sim x}\mathbf 1_{\{\eta_{s-}(y)=1\}}.
\end{equation*}
Taking expectations in the compensator identity gives
\begin{equation*}
    \rho(t) = \mathbb E[I_x(t-1,t]] = \lambda\sum_{y\sim x}\int_{t-1}^{t} \mathbb P(\eta_{s-}(y)=1,\eta_{s-}(x)=0)\,ds.
\end{equation*}

By translation invariance, this becomes \(2d\lambda\int_{t-1}^{t}\rho_{10}(s)\,ds\). 

Since \(0\leq\rho_{10}\leq1\), the integral representation shows directly that \(\rho\) is locally Lipschitz, hence locally absolutely continuous, on \((1,\infty)\). The fundamental theorem of calculus therefore gives
\[
\rho'(t)=2d\lambda\rho_{10}(t)-2d\lambda\rho_{10}(t-1)
\]
for Lebesgue-a.e. \(t>1\).

Recall that every initially infected site recovers at time \(1\); hence
\(\rho(t)\geq\rho_0>0\) for \(t<1\). At \(t=1\), fix \(y\sim x\) and let
\(
    E:=\{\eta_0(x)=0,\ \eta_0(y)=1,\
    \mathcal N^{y,x}((0,1))\geq1\}.
\) 
On \(E\), the first arrow from \(y\) to \(x\) occurs at some
\(s\in(0,1)\) while \(y\) is infected. At that time, \(x\) is either
already infected or becomes infected. Since \(x\) was initially healthy,
in either case its infection began after time \(0\) and, under \(p=0\),
lasts beyond time \(1\). Hence \(E\subseteq\{\eta_1(x)=1\}\), and independence gives
\[
    \rho(1)\geq\mathbb P(E)=\rho_0(1-\rho_0)(1-e^{-\lambda})>0.
\]
For the upper bound, if \(x\) is initially healthy and receives no
incoming arrow in \((0,t]\), then \(\eta_t(x)=0\). This event has
probability \((1-\rho_0)e^{-2d\lambda t}>0\), so \(\rho(t)<1\) for every
finite \(t\).

For fixed \(s\), the law at time \(s\) is a translation-ergodic image
of the initial product field and the independent graphical field.  If
\(\rho_{10}(s)=0\), lattice symmetry gives zero probability of either
orientation of a nearest-neighbor disagreement, while the converse is immediate, so 
\[
\rho_{10}(s)=0\iff\rho(s)\in\{0,1\}.
\]  
By countability and connectedness of \(\mathbb Z^d\), the configuration is then almost surely
spatially constant; ergodicity and the already proved inequality
\(\rho(s)<1\) force \(\rho(s)=0\).

Now suppose that $\rho(t_0)=0$ for some $t_0>1$. The integral identity implies that $\rho_{10}(s)=0$ for Lebesgue-a.e. $s\in(t_0-1,t_0)$, and hence the preceding spatial ergodicity argument yields $\rho(s)=0$ for almost every such $s$. However, we established above that $\rho$ is locally absolutely continuous, and thus continuous, on $(1, \infty)$. Being continuous and equal to zero almost everywhere on $(t_0-1, t_0)$, $\rho$ must be identically zero on the entire closed interval $[t_0-1, t_0]$ (intersected with $(1, \infty)$). Applying this property iteratively backward in steps of length 1, we force $\rho(s)=0$ for some $s \in (0, 1]$, which directly contradicts the already established lower bound $\rho(s) \geq \rho_0 > 0$ for $s < 1$. Therefore, $\rho(t)>0$ for every finite $t$.
\end{proof}

\begin{remark}[Density for \(p > 0\)]
  We emphasize that this delayed balance identity relies on the non-resetting rule. For \(p>0\), a vertex recovers at time \(t\) only if its current deadline was set at time \(t-1\) and no subsequent effective attempt was accepted. Thus recovery depends on the marked history over the entire interval \((t-1,t]\), and the one-window argument used above no longer yields a closed identity for \(\rho\).
\end{remark}

\section*{Acknowledgements}
D.A.L. is supported by a S\~ao Paulo Research Foundation (FAPESP) postdoctoral fellowship (grant 2025/02013-0), and N.L.G. is partially supported by CNPq (grant 306496/2024-0). In addition, D.A.L. and N.L.G. acknowledge support from the FAPESP Thematic Project ``Modelagem de sistemas estocásticos'' (grant 2023/13453-5).

\end{document}